\documentclass{amsart}
\numberwithin{equation}{section}

\usepackage{bbm}
\usepackage{comment}
\usepackage{verbatim}
\usepackage{mathrsfs}
\usepackage{mathtools}
\usepackage{latexsym}
\usepackage{stmaryrd}
\usepackage{epsfig}
\usepackage{amsmath}
\usepackage[usenames,dvipsnames]{xcolor}
\usepackage{bbm}
\usepackage{graphics}
\usepackage[utf8]{inputenc}
\usepackage{amssymb}
\usepackage{amsthm}
\usepackage[alphabetic]{amsrefs}
\usepackage{amsopn}
\usepackage{amscd}
\usepackage[all,knot]{xy}
\xyoption{all}
\usepackage{rotating}
\usepackage{tikz}
\usetikzlibrary{calc}
\usepgflibrary{shapes.geometric}
\usepgflibrary{shapes.misc}
\usetikzlibrary{positioning}
\usetikzlibrary{decorations}
\usetikzlibrary{decorations.pathreplacing}
\usepackage{hyperref}
\usepackage{tikz}
\usepackage{tikz-cd}
\usepackage{microtype} 
\mathtoolsset{showonlyrefs}

\theoremstyle{plain}
\newtheorem{thm}{Theorem}[subsection]
\newtheorem{lem}[thm]{Lemma}
\newtheorem{prop}[thm]{Proposition}

\newtheorem*{thm*}{Theorem}
\newtheorem*{lem*}{Lemma}
\newtheorem*{prop*}{Proposition}
\newtheorem*{cor*}{Corollary}

\theoremstyle{definition}

\newtheorem*{defn*}{Definition}

\newtheorem{ex}[thm]{Example}
{}
\newtheorem{rem}[thm]{Remark}
\newtheorem*{rem*}{Remark}
\newtheorem*{war*}{Warning}

\newtheorem*{hyp_plain*}{Hypotheses}
\newtheorem{notation}[thm]{Notation}{}
{}
{}
{}
{}

\theoremstyle{remark}
{}
{}
{}

\def\to{\longrightarrow} 

\def\PP{\mathbb{P}}

\def\ZZ{\mathbb{Z}}
\DeclareUnicodeCharacter{221E}{$\infty$}

\def\sfC{\mathsf{C}}
\def\sfD{\mathsf{D}}

\def\mfr{\mathfrak{r}}

\def\op{\mathrm{op}}

\DeclareMathOperator{\Hom}{Hom}

\DeclareMathOperator{\id}{id}

\DeclareMathOperator{\modu}{\mathsf{mod}}
\DeclareMathOperator{\Modu}{\mathsf{Mod}}

\DeclareMathOperator{\Gr}{\mathsf{Gr}}
\DeclareMathOperator{\gr}{\mathsf{gr}}

\DeclareMathOperator{\sGr}{\underline{\mathsf{Gr}}}

\DeclareMathOperator{\Inj}{\mathsf{Inj}}
\DeclareMathOperator{\Proj}{\mathsf{Proj}}
\DeclareMathOperator{\proj}{\mathsf{proj}}

\DeclareMathOperator{\RHom}{\mathbf{R}Hom}

\DeclareMathOperator{\Ext}{Ext}

\definecolor{internationalkleinblue}{rgb}{0.0, 0.18, 0.65}

\newcommand{\shortminus}{\scalebox{0.75}[1.0]{\( - \)}}

\theoremstyle{theorem}
\newtheorem{ThmIntro}{Theorem}

\theoremstyle{definition}

\makeatletter
\@namedef{subjclassname@2020}{%
  \textup{2020} Mathematics Subject Classification}
\makeatother
\usepackage{float}
\usepackage[bottom]{footmisc}

\title{Tilting in $Q$-shaped derived categories}

\author{Sira Gratz}
\address{Sira Gratz, Aarhus University, Department of Mathematics, Ny Munkegade 118, bldg. 1530
DK-8000 Aarhus C, Denmark
}
\email{sira@math.au.dk}

\author{Henrik Holm}
\address{Henrik Holm, Department of Mathematical Sciences, University of Copenhagen, Universitetsparken 5,
2100 Copenhagen \O, Denmark
}
\email{holm@math.ku.dk}

\author{Peter J{\o}rgensen}
\address{Peter J{\o}rgensen, Aarhus University, Department of Mathematics, Ny Munkegade 118, bldg. 1530
DK-8000 Aarhus C, Denmark
}
\email{peter.jorgensen@math.au.dk}

\author{Greg Stevenson}
\address{Greg Stevenson, Aarhus University, Department of Mathematics, Ny Munkegade 118, bldg. 1530
DK-8000 Aarhus C, Denmark
}
\email{greg@math.au.dk}


\begin{document}
\begin{abstract}

\noindent 
The main result of this paper is that there is sometimes a triangulated equivalence between $\sfD_Q( A )$, the $Q$-shaped derived category of an algebra $A$, and $\sfD( B )$, the classic derived category of a different algebra $B$.  By construction, $\sfD_Q( A )$ consists of $Q$-shaped diagrams of $A$-modules for a suitable small category $Q$.  Our result concerns the case where $Q$ consists of shifts of indecomposable projective modules over a self-injective $\ZZ$-graded algebra $\Lambda$.

\medskip
\noindent
A notable special case is the result by Iyama, Kato, and Miyachi that $\sfD_N( A )$, the $N$-derived category of $A$, is triangulated equivalent to $\sfD\big( T_{ N-1 }( A ) \big)$, the classic derived category of $T_{ N-1 }( A )$, which denotes upper diagonal $( N-1 ) \times ( N-1 )$-matrices over $A$.  Several other special cases will also be discussed.

\end{abstract}
%

\maketitle






\section{Introduction}

The main result of this paper, Theorem \ref{thm:A} below, is that there is sometimes a triangulated equivalence 
\[
  \sfD_Q( A ) \cong \sfD( B ).
\]
On the left hand side, $\sfD_Q( A )$ is the $Q$-shaped derived category of an algebra $A$ which by construction consists of $Q$-shaped diagrams of $A$-modules for a suitable small category $Q$.  On the right hand side, $\sfD( B )$ is the classic derived category of a different algebra $B$.  Theorem \ref{thm:A} concerns the case where $Q$ consists of shifts of indecomposable projective modules over a self-injective $\ZZ$-graded algebra $\Lambda$.

A notable special case is the following triangulated equivalence, originally due to \cite[prop.\ 4.11]{IKM} with a precursor appearing in \cite[thm.\ 3.1]{Kajiura-Saito-Takahashi}.
\begin{equation}
\label{equ:N-derived}
  \sfD_N( A ) \cong \sfD\big( T_{ N-1 }( A ) \big)
\end{equation}
On the left hand side, $\sfD_N( A )$ is the derived category of $N$-complexes of $A$-modules in which the composition of any $N$ consecutive differentials is zero; see \cite[def.\ 3.6]{IKM}.  On the right hand side, $T_{ N-1 }( A )$ is upper diagonal $( N-1 ) \times ( N-1 )$-matrices over $A$.  

The category $\sfD_Q( A )$ was defined and investigated in \cite{HJ-jlms} and \cite{HJ-tams} building on ideas by Iyama and Minamoto \cite[sec.\ 2]{IM}; see \cite{HJ-intro} for a brief introduction.  
\begin{figure}
\begin{tikzpicture}[scale=1]
  \node at (-3,0){$\cdots$};
  \draw[->] (-2.71,0) to (-2.2,0);
  \node at (-2,0){$2$};
  \draw[->] (-1.8,0) to (-1.2,0);  
  \node at (-1,0){$1$};
  \draw[->] (-0.8,0) to (-0.2,0);    
  \node at (0,0){$0$};
  \draw[->] (0.2,0) to (0.65,0);    
  \node at (1,0){$-1$};
  \draw[->] (1.30,0) to (1.67,0);    
  \node at (2,0){$-2$};    
  \draw[->] (2.32,0) to (2.66,0);    
  \node at (3,0){$\cdots$};    
\end{tikzpicture}
\caption{The category underlying chain complexes and $N$-complexes is given by this diagram with suitable relations.}
\label{fig:linear_category}
\end{figure}
If $Q$ is given by Figure \ref{fig:linear_category} with the relations that any two consecutive arrows compose to zero, then $\sfD_Q( A )$ is {\em equal} to the classic derived category $\sfD( A )$.  On the other hand, if $Q$ is given by a single vertex and a loop squaring to zero, then $\sfD_Q( A )$ is the derived category of differential $A$-modules, which is very far from being equivalent to a classic derived category.  This paper concerns an intermediate situation which we now describe.  The following setup is due to Yamaura \cite[sec.\ 3.1]{Yamaura}.
\begin{itemize}
\setlength\itemsep{4pt}

\item  $k$ is an algebraically closed field.

\item  $\Lambda$ is a finite dimensional self-injective $k$-algebra which is $\ZZ$-graded, concentrated in non-negative degrees, and has $\Lambda_0$ of finite global dimension.

\item  $T$ is the graded right module $\bigoplus_{i=0}^{\ell-1} \Lambda(i)_{\leq 0}$ over $\Lambda$ where $\ell$ is the maximal degree in which $\Lambda$ is non-zero.

\item  $P_1$, $\ldots$, $P_n$ are representatives up to isomorphism of the indecomposable projective right modules over $\Lambda$.  The $P_i$ are canonically $\ZZ$-graded.

\item  $Q$ is the category $\{P_i(j) \mid 1\leq i \leq n \text{ and } j\in \ZZ\}$ where $(j)$ denotes the $j$'th graded shift.

\end{itemize}
Under this setup, $Q$ is a category for which the $Q$-shaped derived category makes sense, see Remark \ref{rem:Q_works}, and we prove the following.

\begin{ThmIntro}
[=Theorem \ref{thm:tilting}]
\label{thm:A}
Setting $\Gamma = \Hom_{\sGr \Lambda}(T,T)$, 
where $\sGr \Lambda$ is the stable category of $\ZZ$-graded right modules over $\Lambda$, we have the triangulated equivalence
\[
\sfD_Q(A) \cong \sfD(\Gamma\otimes_k A).
\]
\end{ThmIntro}

\noindent
The theorem is established by finding a tilting object in $\sfD_Q( A )$ with endomorphism ring $\Gamma\otimes_k A$.  This builds on \cite[prop.\ 3.3]{Yamaura}, which shows that $T$ is a tilting object of $\sGr \Lambda$.  Jasso recently provided an alternative proof of the theorem, see \cite[cor.\ 3.28 and exa.\ 3.32]{Jasso}.

To obtain the equivalence \eqref{equ:N-derived} as a special case of Theorem \ref{thm:A}, set $\Lambda = k[ X ]/( X^N )$ with $X$ in degree $1$.  Then $Q$ is the category given by Figure \ref{fig:linear_category} with the relations that any $N$ consecutive arrows compose to zero, so $\sfD_Q( A )$ is $\sfD_N( A )$, the derived category of $N$-complexes.  Moreover, $\Gamma$ is $T_{ N-1 }( k )$ so $\Gamma \otimes_k A$ is $T_{ N-1 }( A )$ and Theorem \ref{thm:A} gives \eqref{equ:N-derived}.  Different choices of $\Lambda$ provide many other small categories $Q$, and several other special cases of Theorem \ref{thm:A} will be discussed in Section \ref{sec:Examples}.  Section \ref{sec:prelim} contains some preliminary material, and Section \ref{sec:Q} proves Theorem \ref{thm:A} (= Theorem \ref{thm:tilting}).

\medskip
\noindent
{\bf Acknowledgement.}
This work was supported by the EPSRC (grant EP/V038672/1), VILLUM FONDEN (grant VIL42076), a DNRF Chair from the Danish National Research Foundation (grant DNRF156), a Research Project 2 from the Independent Research Fund Denmark (grant 1026-00050B), and by Aarhus University Research Foundation (grant AUFF-F-2020-7-16).

\section{Preliminaries}\label{sec:prelim}


\subsection{Graded rings and modules}

Let $R$ be a graded ring, by which we mean a $\ZZ$-graded ring. We denote by $\Modu R$ and $\Gr R$ the categories of right modules and graded right modules respectively. We use $\modu R$ and $\gr R$ to denote their respective subcategories of finitely presented objects and finitely presented projective objects in these categories are denoted by $\proj R$ and $\proj^\ZZ R$.

Given a graded $R$-module $M$ we can form the shifted module $M(i)$ with $M(i)_j = M_{i+j}$. The category $\Gr R$ is enriched in graded abelian groups via, for graded $R$-modules $M$ and $N$,
\[
\hom_{\Gr R}(M,N) = \bigoplus_{i\in \ZZ} \Hom_{\Gr R}(M, N(i)).
\]

In the special case that $R$ is non-negatively graded, i.e.\ $R = \bigoplus_{i\geq 0} R_i$ then given $M\in \Gr R$ the subgroup $M_{\geq n} = \bigoplus_{i\geq n} M_i$ is a graded submodule for any $n\in \ZZ$. This gives rise to truncation functors
\[
\xymatrix{
\Gr R \ar[rr]<0.75ex>^-{(-)_{\geq n}} \ar[rr]<-0.75ex>_-{(-)_{\leq n}} && \Gr R
}
\]
where $M_{\geq n} = \bigoplus_{i\geq n} M_i$ and $M_{\leq n} = M/M_{\geq n+1}$, which correspond to the right and left adjoints respectively
\[
\xymatrix{
\Gr_{\geq n} R \ar[rr]<0.75ex> \ar@{<-}[rr]<-0.75ex>_-{(-)_{\geq n}} && \Gr R
}
\text{ and }
\xymatrix{
\Gr_{\leq n} R \ar@{<-}[rr]<0.75ex>^-{(-)_{\leq n}} \ar[rr]<-0.75ex> && \Gr R
}
\]
for the inclusions of the full subcategories $\Gr_{\geq n} R$ and $\Gr_{\leq n} R$ consisting of those graded modules living in the obvious degrees (we allow ourselves here to use the same notation for these functors even though they have different targets). 


\subsection{Our standing hypotheses and setup}\label{ssec:hyp}

Let us now fix the setting we will work in. Throughout we work over an algebraically closed field $k$. The assumption that $k$ is algebraically closed is only required to apply \cite{Yamaura}*{Theorem~3.3}, i.e.\ the other results are valid over any field $k$. We let $\Lambda$ be a finite dimensional graded $k$-algebra, $\{e_0,\ldots,e_n\}$ a complete set of primitive orthogonal idempotents for $\Lambda_0$, and we set
\[
\ell = \sup \Lambda = \max \{i \mid \Lambda_i\neq 0\}.
\]
We assume that $\Lambda$ is concentrated in non-negative degrees. 

We write the simple module corresponding to $e_i$ as $S_i$ and it has projective cover $P_i$. The simples and projectives are naturally graded modules and we view them as such. We note that the $S_i$ are concentrated in degree $0$ and, restricting the action of $\Lambda$ to $\Lambda_0$, are precisely the simple $\Lambda_0$-modules. Moreover, $P_i$ is the base change of the corresponding projective $\Lambda_0$-module to $\Lambda$.

Associated to $\Lambda$ is the \emph{companion category} $\sfC$ with set of objects $\ZZ$ and morphisms $\sfC(i,j) = \Lambda_{j-i}$ with composition induced by the multiplication on $\Lambda$. The idempotent completion of $\sfC$ is the full subcategory 
\[
Q = \{P_i(j) \; \mid \; 1\leq i \leq n \text{ and } j\in \ZZ\}
\]
of $\Gr \Lambda$ and both of these categories are Morita equivalent to $\proj^\ZZ \Lambda$ i.e.\
\[
\Gr \Lambda  = \Modu \sfC = \Modu Q.
\]
Let us denote by $\mfr$ the radical of $Q$ (in the sense of \cite{Street}), which is also the radical of $\Lambda$ interpreted through the above lens. We denote the $k$-dual of a vector space $V$ by $V^\ast$, in particular we denote the dual of $\Lambda$ by $\Lambda^\ast$.

\begin{prop}\label{prop:HJconditions1}
The category $Q$ has the following properties:
\begin{itemize}
\item[(1)] it is $k$-linear and all morphism spaces are finite dimensional;
\item[(2)] it is locally bounded, i.e.\ for each $q\in Q$ the functors $Q(q,\shortminus)$ and $Q(\shortminus,q)$ are non-zero on only finitely many objects;
\item[(3)] given $q\neq q'\in Q$ we have a vector space decomposition $Q(q,q) = k\cdot\id_q \oplus \mfr(q,q)$ and $Q(q',q)\circ Q(q,q') \subseteq \mfr(q,q)$;
\item[(4)] there is an $N$ such that $\mfr^N = 0$;
\item[(5)] if $\Lambda$ is self-injective then $Q$ has a Serre functor given by $S = (\shortminus)\otimes_\Lambda \Lambda^\ast$, i.e.\ there is a natural isomorphism $Q(\shortminus,?) \cong \Hom_k(Q(?,S(\shortminus)),k)$.
\end{itemize}
\end{prop}
\begin{proof}
Viewing $Q$ as the idempotent completion of $\sfC$ the first four statements are evident. After all, $\Lambda$ is a finite dimensional algebra over $k$ and so $\Lambda_i = 0$ for $\vert i \vert \gg 0$ which gives (1) and (2), as well as (4) since $\mfr$ is identified with the Jacobson radical of $\Lambda$ (we note that the Jacobson radical is automatically gradable in our setting, see for instance \cite{GG}*{Corollaries~4.4 and 4.5}). Again, since $\Lambda$ is non-negatively graded, for $q\neq q'$ the vector space $Q(q,q')$ consists of homogeneous elements of positive degree and so is certainly contained in the radical and so (3) follows too.

Now let us suppose that $\Lambda$ is self-injective. Then $S$ is a well-defined endofunctor of $Q$ and is an equivalence. In fact, $S$ is an autoequivalence on all of $\proj^\ZZ \Lambda$. In this latter category we have
\[
\hom(\Lambda, S(\Lambda(i))) = \hom(\Lambda, \Lambda^\ast(i)) \cong \Lambda^\ast(i)
\]
and taking degree $0$ components gives
\[
\Hom_{\Gr \Lambda}(\Lambda, S(\Lambda(i))) = (\Lambda^\ast)_i = (\Lambda_{-i})^\ast = \Hom_{\Gr \Lambda}(\Lambda(i), \Lambda)^\ast.
\]
This natural isomorphism then extends to arbitrary finitely generated projectives by additivity.
\end{proof}

\begin{rem}
\label{rem:Q_works}
The point is that, provided $\Lambda$ is self-injective, $Q$ satisfies \cite{HJ-intro}*{Setup~1.1} and so we are in a situation where we can consider the $Q$-shaped derived category.
\end{rem}


\subsection{Change of base}

We continue with the setup as above, so $\Lambda$ is a non-negatively graded $k$-algebra and we associate to it categories $\sfC$ and $Q$.

Let $A$ be a $k$-algebra. Then we can consider the graded ring $\Lambda \otimes_k A$, where the grading is induced by the one from $\Lambda$. Given a small $k$-linear category $\sfD$ we let $\Modu \sfD$ denote the category of right $\sfD$-modules
\[
\Modu \sfD = [\sfD^\op, \Modu k]_k
\]
i.e.\ the category of $k$-linear functors from $\sfD^\op$ to $\Modu k$. We can form the base change of $\sfD$ to $A$ denoted $\sfD\otimes_k A$ which is the category with the same objects as $\sfD$ and for $d,d'\in \sfD$
\[
(\sfD\otimes_k A)(d,d') = \sfD(d,d')\otimes_k A
\]
with the obvious composition induced by the composition of $\sfD$ and multiplication of $A$. As usual, we will drop the $k$ from the subscript when it is clear (and we feel like it). We have an obvious equivalence
\[
\Modu (\sfD\otimes A) \cong [\sfD^\op, \Modu A]_k
\]
between modules over the base change of $\sfD$ to $A$ and $k$-linear functors from $\sfD^\op$ to $\Modu A$ (this is just the hom-tensor adjunction).

\begin{rem}
The category $[\sfD^\op, \Modu A]_k$ is denoted by ${}_{\sfD^{\op},A}\Modu$ in \cite{HJ-intro}.
\end{rem}

\begin{lem}
Let $A$ be a $k$-algebra. There are equivalences of categories
\[
\Modu (Q\otimes A) \cong \Gr(\Lambda \otimes A) \cong \Modu (\sfC\otimes A).
\]
\end{lem}
\begin{proof}
The category $\Gr(\Lambda \otimes A)$ has $\proj^\ZZ (\Lambda\otimes A)$ as a set of finitely presented projective generators and so
\[
\Gr(\Lambda \otimes A) \cong \Modu(\proj^\ZZ (\Lambda\otimes A)).
\]
Now $\proj^\ZZ (\Lambda\otimes A)$ is Morita equivalent to the full subcategory consisting of the rank $1$ free modules $\Lambda\otimes A(i)$ for $i\in \ZZ$ (the former being the additive closure of the latter). We have
\[
\Hom_{\Gr(\Lambda \otimes A)}(\Lambda\otimes A(i), \Lambda\otimes A(j)) \cong (\Lambda\otimes A)_{j-i} = \Lambda_{j-i}\otimes A = (\sfC\otimes A)(i,j)
\]
i.e.\ the category of rank $1$ free modules is nothing but $\sfC \otimes A$, which gives the second equivalence in the statement. Since $\sfC$ and $Q$ are Morita equivalent so are $Q\otimes A$ and $\sfC\otimes A$ and so we are done.
\end{proof}

\begin{notation}\label{notation:adjoints}
The structure map $k\to A$ induces a functor $Q = Q\otimes k \stackrel{i}\to Q\otimes A$. This gives rise to the base change functor
\[
i^*\colon \Modu Q \to \Modu Q\otimes A
\]
which is just, up to the canonical equivalences, the usual base change functor induced by the map $\Lambda \to \Lambda \otimes A$ of graded rings. As usual $i^*$ has adjoints $i^* \dashv i_* \dashv i^!$ which automatically, by the above, are similarly compatible with the equivalence to graded modules. We emphasize that, since we are working over a field $k$, the base change functor $i^*$ is exact\textemdash{}indeed, $\Lambda\otimes A$ is free over $\Lambda$.
\end{notation}

We understand completely the maps between objects of $\Gr (\Lambda\otimes A)$ coming from $\gr \Lambda$.

\begin{lem}\label{lem:bc1}
Let $M$ be a finitely presented graded $\Lambda$-module. There is a natural isomorphism $\hom_{\Lambda\otimes A}(i^*(M),i^*(-)) \cong \hom_\Lambda(M,-)\otimes A$. In particular, taking the $0$th degree pieces gives the analogous statement for the usual morphisms.
\end{lem}
\begin{proof}
One reduces, as usual, to checking on finite rank free modules where the statement is clear. 
\end{proof}




\section{The $Q$-shaped derived category}
\label{sec:Q}

We refer to the survey \cite{HJ-intro} for an introduction to $Q$-shaped derived categories. In particular, we follow Section~2 of loc.\ cit.\ to give an interpretation of these categories which will be convenient for us.

We continue with our standard setup as in Section~\ref{ssec:hyp}, with the additional hypothesis that $\Lambda$ is self-injective. For $q\in Q$ we denote by $S_q$ (this would be $S\{q\}$ in the notation of \cite{HJ-intro}) the simple module concentrated at $q$. Noting that $q = P_i(j)$ for some $1\leq i \leq n$ and $j\in \ZZ$ one sees that under the equivalence $\Modu Q = \Gr \Lambda$ the simple $S_q$ is just $S_i(j)$. We are led to consider the class, for a $k$-algebra $A$,
\[
\mathscr{E} = \mathscr{E}_{A} = \{M\in \Modu (Q\otimes A) \mid \Ext_Q^1(S_q, M)=0 \; \forall \; q\in Q\}
\]
where we consider $M$ as a $Q$-module via the restriction functor $i_*$ induced by $i\colon Q\to (Q\otimes A)$ as in Notation~\ref{notation:adjoints}, and we drop the dependence on our fixed $k$-algebra $A$ from the notation as indicated. 


\subsection{Base change and generators}

\begin{lem}
There are equalities
\begin{align*}
\mathscr{E} &= \{M\in \Modu (Q\otimes A) \mid i_*M \in \Proj^\ZZ \Lambda\} \\
&= \{M\in \Modu (Q\otimes A) \mid i_*M \in \Proj \Lambda\}
\end{align*}
\end{lem}
\begin{proof}
By \cite{HJ-jlms} Theorem~7.1 and Definition~4.1 the class $\mathscr{E}$ consists of those modules whose projective dimension over $Q$ is finite. This is just asking for finite projective dimension as a graded $\Lambda$-module and, because $\Lambda$ is self-injective, this implies projectivity. One then notes that projectivity is independent of whether or not we consider the grading (cf.\ \cite{NV}*{Corollary~3.3.7}).
\end{proof}

We can consider the cotorsion pair that $\mathscr{E}$ generates and that it cogenerates. We denote the right and left $\Ext^1$-perpendiculars of $\mathscr{E}$ using $\perp_1$, e.g.\
\[
\mathscr{E}^{\perp_1} = \{N \in \Modu (Q\otimes A) \mid \Ext^1(M,N)=0 \; \forall \; M\in \mathscr{E}\}.
\]
By \cite{HJ-jlms}*{Theorem~6.5} both ${}^{\perp_1}\mathscr{E}$ and $\mathscr{E}^{\perp_1}$ are Frobenius categories. As in \cite{HJ-intro}*{Section~4} we can describe the $Q$-shaped derived category as
\[
\underline{{}^{\perp_1}\mathscr{E}} = \sfD_Q(A) = \underline{\mathscr{E}^{\perp_1}}
\]
where the underline denotes, as usual, the stable category. In particular, $\sfD_Q(A)$ is algebraic and the above provides a pair of differential graded enhancements.

\begin{rem}
If we take $A=k$ so $(Q\otimes A) = Q$ then we see that $\mathscr{E} = \Proj^\ZZ \Lambda = \Inj^\ZZ \Lambda$ is the full subcategory of projective, and equivalently injective, graded modules. Hence $\mathscr{E}^{\perp_1} = \Gr\Lambda = {}^{\perp_1}\mathscr{E}$. In particular, $\sfD_Q(k) = \sGr \Lambda$.
\end{rem}

\begin{lem}
For every $M\in \Modu Q$ we have $i^*M \in {}^{\perp_1}\mathscr{E}$. 
\end{lem}
\begin{proof}
Let $N\in \mathscr{E}$ i.e.\ $i_*N$ is projective, and hence injective, over $Q$. So we have
\[
\RHom_{Q\otimes A}(i^*M, N) \cong \RHom_Q(M, i_*N) = \Hom_Q(M, i_*N)
\]
is concentrated in degree $0$ (recall that $i^*$ is exact). 
\end{proof}

\begin{prop}
The base change functor $i^*\colon \Modu Q \to \Modu (Q\otimes A)$ induces a triangulated coproduct preserving functor $i^*\colon \sGr \Lambda \to \sfD_Q(A)$.
\end{prop}
\begin{proof}
By the previous lemma the image of $i^*$ is contained in ${}^{\perp_1}\mathscr{E}$. As always, $i^*$ sends projectives to projectives and, as we have already noted, it is exact. Thus it descends to an exact functor between the stable categories. It only remains to make the identification $\underline{{}^{\perp_1}\mathscr{E}} = \sfD_Q(A)$.  
\end{proof}

\begin{rem}\label{rem:adjoint}
It follows that $i^*\colon \sGr \Lambda \to \sfD_Q(A)$ has a right adjoint. We observe that it is induced by stabilizing $i_*$. Indeed, we can of course restrict $i_*$ to ${}^{\perp_1}\mathscr{E}$ and $i_*$ sends projectives to projectives.
\end{rem}

Let us now understand something about the image of $i^*$. 

\begin{lem}\label{lem:bc2}
The functor $i^*$ preserves compacts and the image of $i^*$ generates $\sfD_Q(A)$.
\end{lem}
\begin{proof}
As observed in Remark~\ref{rem:adjoint} we have an explicit description of the right adjoint of $i^*$ as the functor induced on the stable categories by $i_*$. Since $i_*$ preserves coproducts so does the functor it descends to. Hence $i^*$ preserves compacts. By \cite{HJ-tams} the category $\sfD_Q(A)$ is compactly generated by the set of objects $\{i^*S_q \mid q\in Q\}$ and so the image of $i^*$ generates.
\end{proof}


\subsection{Tilting}

In this section we will make the additional assumption that the ring $\Lambda_0$ has finite global dimension. We can then consider the graded $\Lambda$-module
\[
T = \bigoplus_{i=0}^{\ell-1} \Lambda(i)_{\leq 0},
\]
where $\ell$ is the maximal degree in which $\Lambda$ is non-trivial, which Yamaura \cite{Yamaura} has shown gives a tilting object for $\sGr \Lambda$. This induces a tilting object for the $Q$-shaped derived category of any $k$-algebra $A$.

\begin{notation}
Given chain complexes $X$ and $Y$ with values in some $k$-linear abelian category we will denote the hom-complex, which is a complex of $k$-vector spaces, by $\hom(X,Y)$.
\end{notation}

\begin{thm}\label{thm:tilting}
The object $i^*T\in \sfD_Q(A)$ is a tilting object with endomorphism ring $\Hom_{\sGr \Lambda}(T,T)\otimes_k A$. Thus, setting $\Gamma = \Hom_{\sGr \Lambda}(T,T)$, we have that
\[
\sfD_Q(A) \cong \sfD(\Gamma\otimes_k A)
\]
is the derived category of a ring.
\end{thm}

\begin{proof}
As mentioned, by Yamaura's work, we already know that $T$ is a compact generator for $\sGr \Lambda$. We then learn from Lemma~\ref{lem:bc2} that $i^*T$ is a compact generator for $\sfD_Q(A)$ and so it remains to compute the derived endomorphism ring $\RHom_{\sfD_Q(A)}(i^*T,i^*T)$. We need to show that $\RHom_{\sfD_Q(A)}(i^*T,i^*T)$ has cohomology precisely $\Gamma\otimes A$ in degree $0$.

Let $\widetilde{T}$ be a complete graded projective-injective resolution for $T$ over $\Lambda$, which is moreover chosen to be degree-wise finitely presented. We already know that the complex $\hom(\widetilde{T}, T)$, which computes the cohomology of the derived endomorphism ring in $\sGr \Lambda$, has cohomology $\Hom_{\sGr\Lambda}(T,T)$ concentrated in degree $0$.

Now let us apply the tensor product functor. Because $-\otimes_k A$ is exact the complex $i^*\widetilde{T}$ remains an exact complex of graded projectives over $\Lambda\otimes A$. Thus it is a complete projective resolution for $i^*T$ in the Frobenius category ${}^{\perp_1}\mathscr{E}$. In particular, we may compute the cohomology groups $\Hom_{\sfD_Q(A)}(i^*T, \Sigma^j i^*T) = \Hom_{\underline{{}^{\perp_1}\mathscr{E}}}(i^*T, \Sigma^ji^*T)$ of the derived endomorphism ring of $i^*T$ in $\sfD_Q(A)$ by the cohomology of the hom-complex $\hom(i^*\widetilde{T}, i^*T)$ as follows
\[
\mathrm{H}^*\hom(i^*\widetilde{T}, i^*T) = \mathrm{H}^*\hom(\widetilde{T}, T)\otimes_k A \cong \Hom_{\sGr \Lambda}(T,T)\otimes_k A
\]
where the first identification is an instance of Lemma~\ref{lem:bc1}, using that $i^*T$ is concentrated in degree $0$ and $\widetilde{T}$ has finitely presented terms. This completes the proof of the theorem.
\end{proof}

\begin{rem}
In fact, in the proof above $i^*\widetilde{T}$ remains totally acyclic: for any graded projective $\Lambda\otimes A$-module $P$ we have
\[
\hom(i^*\widetilde{T}, P) \cong \hom(\widetilde{T}, i_*P)
\]
which is acyclic, because $i_*P$ is projective over $\Lambda$.
\end{rem}



\section{Examples}
\label{sec:Examples}

There are many examples where one can apply Theorem~\ref{thm:tilting}. We content ourselves here with presenting a couple of illustrative cases (some of which inspired this work).


\subsection{Mesh categories of type $A$}

The graded ring $\Lambda = \Lambda_n$ in question is the preprojective algebra of type $A_n$, which is graded by giving the original copy of $kA_n$ degree $0$ and the doubled arrows degree $1$. That is to say, we have

\[
\xymatrix{
1 \ar[r]<0.5ex>^-{\alpha_1} \ar@{<-}[r]<-0.5ex>_-{\beta_1} & 2 \ar[r]<0.5ex>^-{\alpha_2} \ar@{<-}[r]<-0.5ex>_-{\beta_2} & \cdots \ar[r]<0.5ex>^-{\alpha_{n-2}} \ar@{<-}[r]<-0.5ex>_-{\beta_{n-2}} & n-1 \ar[r]<0.5ex>^-{\alpha_{n-1}} \ar@{<-}[r]<-0.5ex>_-{\beta_{n-1}} & n
}
\]
with relations $\beta_1\alpha_1$, $\alpha_{n-1}\beta_{n-1}$, and $\beta_{i+1}\alpha_{i+1} - \alpha_{i}\beta_{i}$ for $1\leq i \leq n-2$, where we write paths as compositions of arrows i.e.\ they should be read from right to left, and we assign the $\alpha$'s degree $0$ and the $\beta$'s degree $1$. The category $Q$ then has objects $(i,j)$ for $1\leq i \leq n$ and $j\in \ZZ$ and is precisely the mesh algebra of type $A_n$, i.e.\ it is the full subcategory of indecomposable objects of $\sfD^b(kA_n)$. 

The algebra $\Lambda_n$ satisfies our hypotheses: it is non-negatively graded, self-injective, and $\Lambda_0$ has finite global dimension. Thus, for any $k$-algebra $A$ we have an equivalence
\[
\sfD_{Q_n}(A) \cong \sfD(\Gamma_n\otimes A)
\]
for $\Gamma_n$ some finite dimensional algebra. It turns out, see \cite{Yamaura}*{Section~4.2}, that the algebra $\Gamma_n$ is the Auslander algebra of $kA_{n-1}$ (with linear orientation). Let us spell out explicitly the three smallest examples.

\begin{ex}
The algebra $\Lambda_1$ is just $k$ viewed as a $\ZZ$-graded ring concentrated in degree $0$. Thus $Q_1$ is the $k$-linearization of
\[
\xymatrix{\cdots & i-1 & i & i+1 & \cdots }
\]
i.e.\ we get a category with object set $\ZZ$ and all morphism spaces are trivial except for the endomorphism rings which are all $k$. In this case $\sfD_{Q_1}(A)=0$, for instance taking $A=k$ we have $\sfD_{Q_1}(k) = \sGr k = 0$.
\end{ex}

\begin{ex}
If we take $n=2$ then $\Lambda$ is the graded algebra
\[
\xymatrix{
1 \ar[r]<0.5ex>^-{\alpha} \ar@{<-}[r]<-0.5ex>_-{\beta} & 2 }
\]
with square-zero relations and $\beta$ in degree $1$. Thus $Q_2$ is the $k$-linear category
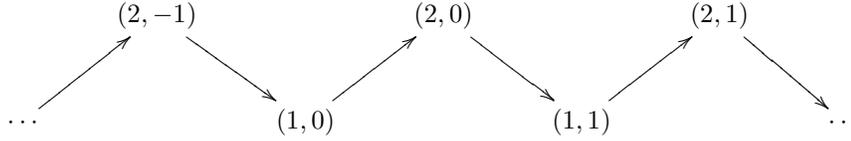
\[
\xymatrix{
& (2,-1) \ar[dr] & & (2,0) \ar[dr] & & (2,1) \ar[dr] \\
\cdots \ar[ur] & & (1,0) \ar[ur] & & (1,1) \ar[ur] & & \cdots
}
\]
again with all square-zero relations. Reindexing, we see that this is the category defining chain complexes, or equivalently it is the category associated to the graded algebra $k[x]/x^2$ with $x$ of degree $1$. In this case $\Gamma_2 = k$ and we obtain $\sfD_{Q_2}(A) \cong \sfD(A)$. This example is discussed extensively in \cite{HJ-intro}. 
\end{ex}

\begin{ex}
If we take $n=3$ then $Q_3$ is the following category
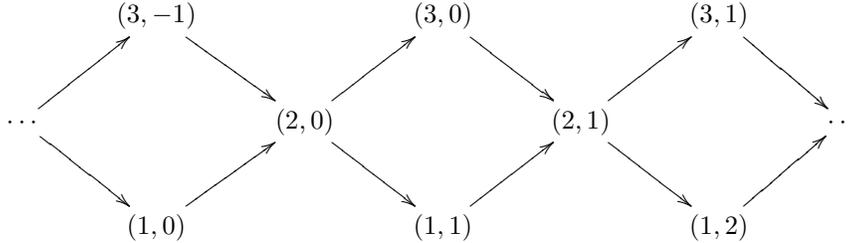
\[
\xymatrix{
& (3,-1) \ar[dr] & & (3,0) \ar[dr] & & (3,1) \ar[dr] & & \\ 
\cdots \ar[ur] \ar[dr] && (2,0) \ar[ur] \ar[dr] & & (2,1) \ar[ur] \ar[dr] & & \cdots \\
& (1,0) \ar[ur] & & (1,1) \ar[ur] & & (1,2) \ar[ur]
}
\]
with the usual mesh relations (that is, squares anticommute and the top and bottom fringes are square-zero). In this diagram, the upward pointing arrows correspond to generators of $\Lambda$ of degree $0$ and the downward arrows correspond to generators of $\Lambda$ of degree $1$. The algebra $\Gamma_3$ is the Auslander algebra of $kA_2$, i.e.\ it is $kA_3/J^2$ where $J$ is the radical. This is derived equivalent to $kA_3$, its Koszul dual, and so we instead work with $kA_3$ below.

The theorem then tells us that $\sfD_{Q_3}(A) = \sfD(kA_3 \otimes A)$. In particular, $\sfD_Q(A)$ has a number of semiorthogonal decompositions into copies of $\sfD(A)$, i.e.\ the subject of the previous example. This could be interpreted as follows: there are numerous ways in which one can find full subcategories of $Q_3$ which are isomorphic to $Q_2$, such that this choice gives rise to a semiorthogonal decomposition of $\sfD_{Q_3}(A)$ into the copy of $\sfD_{Q_2}(A) = \sfD(A)$ induced by fully faithful inclusion of $Q_2$ and the corresponding localization, which is a copy of $\sfD(kA_2 \otimes A)$. Essentially, $Q_3$ is spliced together from copies of $Q_2$ and the different decompositions of $\sfD_{Q_3}(A)$ reflect how these copies of $Q_2$ interact.

One could, or perhaps even should, compare this to the case of $N$-complexes. Recall that $N$-complexes over $A$ are another name for objects of $\Gr A[x]/x^n$ where $x$ has degree $1$. From, for instance, a $4$-complex we can obtain three different chain complexes by writing $x^4$ as $x\cdot x^3$, $x^2\cdot x^2$, and $x^3\cdot x$ and this gives three different notions of cohomology for such a gadget. It turns out these describe the derived category of $4$-complexes: it is equivalent to $\sfD(kA_3\otimes A)$. Through this lens, $Q_3$ is a different way of encoding the same interactions between three chain complexes.
\end{ex}

In general, the Auslander algebra of $A_{n-1}$ is directed and so $\sfD_{Q_n}(A)$ always admits semiorthogonal decompositions into pieces equivalent to $\sfD(A)$.


\subsection{Exterior algebras}

Let $\Lambda_n$ denote the exterior algebra on $n$-generators with its standard grading, and let $Q_n$ denote the associated $k$-linear category. In this case $Q_n$ has object set $\ZZ$, because $\Lambda_n$ is local, and so coincides with the companion category $\sfC_n$. 

The algebra $\Gamma_n = \Hom_{\sGr \Lambda_n}(T,T)$ is, as in \cite{Yamaura}*{Example~3.16}, given by the (algebra corresponding to the) full subcategory of $Q_n$ on $0,\ldots,n-1$. For any $k$-algebra $A$ we get equivalences
\[
\sfD_Q(A) \cong \sfD(\Gamma_n \otimes A) \cong \sfD(\Gamma_n' \otimes A)
\]
where $\Gamma_n'$ is the Beilinson algebra on $n$ vertices, which is given by the quiver
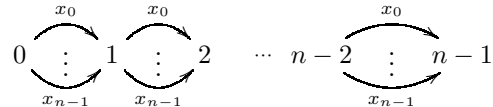
\[
\xymatrix{
0 \ar@/^1pc/[r]^-{x_0} \ar@/_1pc/[r]_-{x_{n-1}} \ar@{}[r]|{\vdots} & 1 \ar@/^1pc/[r]^-{x_0} \ar@/_1pc/[r]_-{x_{n-1}}  \ar@{}[r]|{\vdots} & 2 \ar@{}[r]|{ \cdots } & n-2 \ar@/^1pc/[r]^-{x_{0}} \ar@/_1pc/[r]_-{x_{n-1}}  \ar@{}[r]|{\vdots} & n-1 }
\]
with commutativity relations $x_ix_j = x_jx_i$ at each pair of consecutive vertices. This is just the Koszul dual of the Koszul algebra $\Gamma_n$ given by swapping anticommutativity relations for commutativity ones. This explains the derived equivalence between $\Gamma_n \otimes A$ and $\Gamma_n' \otimes A$. In particular, if $A$ is commutative we obtain an equivalence $\sfD_Q(A) \cong \sfD(\PP^{n-1}_A)$ via the standard tilting object for $\PP^{n-1}_A$ coming from the twisting sheaves (Beilinson's original derived equivalence \cite{Beilinson} over a field generalizes directly to an arbitrary commutative base ring).



\end{document}